\documentclass[11pt,reqno]{amsart}

\usepackage{amsfonts, amsthm, amsmath}
\allowdisplaybreaks[4]

\usepackage{rotating}

\usepackage{tikz}

\usepackage{graphics}

\usepackage{amssymb}

\usepackage{circledsteps}

\usepackage{amsmath}

\usepackage{amscd}

\usepackage[latin2]{inputenc}

\usepackage{t1enc}

\usepackage[mathscr]{eucal}

\usepackage{indentfirst}

\usepackage{graphicx}

\usepackage{graphics}

\usepackage{pict2e}

\usepackage{mathrsfs}

\usepackage{enumerate}

\usepackage[pagebackref]{hyperref}
\hypersetup{colorlinks=true}
\usepackage{cite}
\usepackage{color}
\usepackage{epic}
\usepackage{hyperref} 
\usepackage{framed}
\usepackage{mathabx}

\numberwithin{equation}{section}
\theoremstyle{plain}

\newtheorem{theorem}{Theorem}[section]

\newtheorem{lemma}[theorem]{Lemma}

\newtheorem{proposition}[theorem]{Proposition}

\theoremstyle{definition}

\newtheorem{conj}[theorem]{Conjecture}

\newtheorem{?}[theorem]{Problem}

\newtheoremstyle{named}{}{}{\itshape}{}{\bfseries}{.}{.5em}{#1\thmnote{ #3}}
\theoremstyle{named}

\newcommand{\f}[1]{\ifthenelse{\equal{#1}{1}}{(q;q)_\infty}{(q^{#1};q^{#1})_{\infty}}}

\def\bZ{\mathbb{Z}}

\def\tT{\widetilde{T}}
\def\tZ{\widetilde{Z}}
\def\hZ{\hat{Z}}

\def\ri{\rightarrow}

\begin{document}
\title[Proofs of five conjectural identities on modular rank four Nahm sums]{Proofs of five conjectural identities on modular rank four Nahm sums}

\author[H. Li]{Haijun Li}
\address[Haijun Li]{College of Mathematics and Statistics, Chongqing University, Chongqing 401331, P.R. China}
\email{lihaijun@cqu.edu.cn; lihaijune@163.com}

\date{\today}

\begin{abstract}
Nahm sums and Rogers-Ramanujan type identities have attracted considerable attention in recent years. In this paper, we provide analytic proofs of five conjectural identities on modular rank four Nahm sums that were proposed by Cao and Wang. Moreover, we reveal that the conjectures of Shi-Wang and Cao-Wang are closely related.
\end{abstract}


\maketitle
\section{Introduction}\label{sec:intro}

Rogers-Ramanujan type identities are identities of the sum-to-product type, where the sum side is a $q$-hypergeometric series (possibly involving multiple summations), and the product side is an infinite product. The study of such identities was inspired by the two famous Rogers-Ramanujan identities:
\begin{align}
&\sum_{n\geq 0}\frac{q^{n^2}}{(q; q)_n}=\frac{1}{(q, q^4; q^5)_{\infty}},\quad \sum_{n\geq 0}\frac{q^{n^2+n}}{(q; q)_n}=\frac{1}{(q^2, q^3; q^5)_{\infty}},\label{id:RR}
\end{align}
where we adopt the following standard $q$-series notation (see, e.g.,~\cite{GR90}):
\begin{align*}
&(a; q)_n := \prod_{k=0}^{n-1} (1 - aq^k),\quad  (a; q)_\infty := \prod_{k=0}^{\infty} (1 - aq^k), \quad |q| < 1,\\
&(a_1, \ldots, a_m; q)_n := (a_1; q)_n \cdots (a_m; q)_n, \qquad n \in \mathbb{N} \cup \{\infty\}.
\end{align*}
For brevity we also adopt the notation
\begin{align*}
J_m:=(q^m; q^m)_{\infty}.
\end{align*}
The identities in \eqref{id:RR} were first proved by Rogers \cite{rog94}, later rediscovered by Ramanujan \cite{ram14, ram19}, and independently derived by Schur \cite{sch17}. 

Rogers-Ramanujan type identities play an important role in combinatorics, number theory, Lie algebras, statistical mechanics, and other fields. For a detailed introduction to these identities, we refer the reader to Sills' book~\cite{sil18}.

One of the central problems connecting Rogers-Ramanujan type identities to the theory of modular forms is to find all modular $q$-hypergeometric series. In this regard, Nahm~\cite{nah941, nah942, nah07} considered a particularly important class of series which were usually referred to as {\it Nahm sums}. For a $r\times r$ rational positive definite matrix $A$, a rational vector $B$ of length $r$ and a rational scalar $C$, the rank $r$ Nahm sum corresponding to $(A, B, C)$ is defined as 
\begin{align}
f_{A, B, C}(q):=\sum_{n=(n_1, ..., n_r)^T\in (\bZ_{\geq 0})^r}\frac{q^{\frac{1}{2}n^TAn+n^TB+C}}{(q; q)_{n_1}\cdots (q; q)_{n_r}}.
\end{align}
Nahm's problem is to find all modular Nahm sums $f_{A, B, C}(q)$, and such $(A, B, C)$ is called a {\it modular triple}.

In 2007, Zagier~\cite{zag07} made a significant contribution to Nahm's problem by establishing the conjecture in rank-one case. For higher ranks, namely $r=2$ and $r=3$, he identified numerous candidate modular triples, which were subsequently verified by various authors, including Cao-Rosengren-Wang~\cite{CRW24}, Cherednik-Feigin~\cite{CF13}, Vlasenko-Zwegers~\cite{VZ11}, Wang~\cite{wan241, wan242}, and Zagier~\cite{zag07} himself.

Research related to Nahm's problem is still ongoing. Recently, Cao and Wang \cite{CW25} applied the lift-dual operation to the rank three modular triples in works of Zagier~\cite{zag07} and Cao-Wang~\cite{CW24} to obtain rank four modular triples. In their work, a large number of Rogers-Ramanujan type identities were proved. Alongside the main results obtained in their paper, Cao and Wang also proposed six identities of a conjectural nature. It is precisely the goal of this paper to settle five of them by means of rigorous analytic proofs, thus complementing and extending their work.

Next, we state our main results as follows. The corresponding conjecture numbers in the original paper~\cite{CW25} can be found in parentheses after each theorem.

\begin{theorem}[{cf.~\cite[Conjecture~3.2]{CW25}}]\label{thm:1}
We have
\begin{align}
\sum_{i,j,k,l \geq 0}
\frac{
  q^{\frac{1}{2}(i^2 - i) + j^2 + k^2 + l^2 - il - jk - jl - 2j + 2k + l}
}{
  (q; q)_i (q; q)_j (q; q)_k (q; q)_l
}
&= 8q^{-1} \frac{J_2^6}{J_1^6}, \label{thm11}\\
\sum_{i,j,k,l \geq 0}
\frac{
  q^{i^2 + 2j^2 + 2k^2 + 2l^2 - 2il - 2jk - 2jl - 3j + 3k + l}
}{
  (q^2; q^2)_i (q^2; q^2)_j (q^2; q^2)_k (q^2; q^2)_l
}
&= 2q^{-1} \frac{J_2^3}{J_1^3},\label{thm12} \\
\sum_{i,j,k,l \geq 0}
\frac{
  q^{\frac{1}{2}(i^2 + i) + j^2 + k^2 + l^2 - il - jk - jl - j + k}
}{
  (q; q)_i (q; q)_j (q; q)_k (q; q)_l
}
&= 4 \frac{J_2^6}{J_1^6}.\label{thm13}
\end{align}
\end{theorem}

\begin{theorem}[{cf.\cite[~Conjecture~3.4]{CW25}}]\label{thm:2}
We have
\begin{align}
\sum_{i,j,k,l \geq 0}
\frac{
  q^{i^2 + 2j^2 + 2k^2 + 2l^2 - 2ij + 2ik - 2jk - 2jl + 2i - 2j + 2k}
}{
  (q^2; q^2)_i (q^2; q^2)_j (q^2; q^2)_k (q^2; q^2)_l
}
&= 3 \frac{J_6^3}{J_1 J_2 J_4},\label{thm2}\\
\sum_{i,j,k,l \geq 0} 
\frac{
  q^{\frac{1}{2}(i^2 + i) + j^2 + k^2 + l^2 - ij + ik - jk - jl - j + k+l}
}{
  (q; q)_i (q; q)_j (q; q)_k (q; q)_l
}
&= 3 \frac{J_2 J_3^3}{J_1^4}.\label{thm21}
\end{align}
\end{theorem}

The remainder of this paper is structured as follows. Section \ref{sec:pre} contains some preliminary results that will be used later. In Section \ref{sec:11}, we give the proof of Theorem \ref{thm:1}, and in Section \ref{sec:12}, we present the proof of Theorem \ref{thm:2}. Finally, in Section \ref{sec:conclusion}, we establish a connection between \eqref{thm21} and a conjecture, which appears in \cite{SW25}.

\section{Preliminaries}\label{sec:pre}

In this section, we introduce and prove some preliminary results and concepts that will facilitate the proofs in the subsequent sections. First, we consider Ramanujan's general theta function $f(a, b)$ defined by 
\begin{align}
f(a, b):=\sum_{n=-\infty}^{\infty}a^{n(n+1)/2}b^{n(n-1)/2},\quad |ab|<1.\label{theta}
\end{align}
It satisfies the well-known Jacobi triple product identity~\cite{andtp, ber06}
\begin{align}\label{jacobi}
f(a, b)=(-a, -b, ab, ab)_{\infty}.
\end{align}
Then we recall two important $q$-exponential functions.
\begin{lemma}[{\cite[Appendix II.1-II.2]{GR90}}]\label{lem:0}
We have
\begin{align}
e_q(z)=\sum_{n\geq 0}\frac{z^n}{(q; q)_n}&=\frac{1}{(z; q)_{\infty}},\quad |z|<1,\label{II1}\\
E_q(z)=\sum_{n\geq 0}\frac{q^{\binom{n}{2}}z^n}{(q; q)_n}&=(-z; q)_{\infty}.\label{II2}
\end{align}
\end{lemma}
Recall that the finite $q$-binomial theorem gives
\begin{align}
\sum_{k=0}^Nq^{\frac{k^2+k}{2}}\begin{bmatrix}N\\ k\end{bmatrix}_q=(-q; q)_N,\label{finite-q-bino}
\end{align}
where the $q$-binomial coefficient is
\begin{align*}
\begin{bmatrix}c\\a \end{bmatrix}_q:=\frac{(q; q)_c}{(q; q)_a(q; q)_{c-a}}.
\end{align*}
The next result is a finite Durfee reduction.

\begin{lemma}\label{lem:1}
For every integer $r$ and every $c\geq 0$, we have
\begin{align}
\sum_{k\geq\max(0, -r)}\frac{q^{k(k+r-c)}}{(q; q)_k(q; q)_{k+r}}=\frac{1}{J_1}\sum_{a=0}^cq^{a(r-c+a)}\begin{bmatrix}c\\a \end{bmatrix}_q.\label{lem1}
\end{align}
\end{lemma}

\begin{proof}
For $c=0$, we claim that \eqref{lem1} is the identity corresponding to the Durfee-rectangle decomposition
\begin{align*}
\sum_{k\geq \max(0, -r)}\frac{q^{k(k+c)}}{(q; q)_k(q; q)_{k+r}}=\frac{1}{J_1}.
\end{align*}
Indeed, for $r\geq 0$, this classifies a partition by its maximal $k\times (k+r)$ rectangle. For $r<0$, write $r=-h$ with $h>0$ and set $j=k-h$. The sum becomes
\begin{align*}
\sum_{j\geq 0}\frac{q^{j(j+h)}}{(q; q)_j(q; q)_{j+h}},
\end{align*}
which is the case $h>0$.

Let the left side of \eqref{lem1} be $L_c(r)$, then we have
\begin{align*}
L_{c-1}(r-1)&=\sum_{k\geq \max(0, -r+1)}\frac{q^{k(k+r-c)}(1-q^{k+r})}{(q; q)_k(q; q)_{k+r}}=\sum_{k\geq \max(0, -r)}\frac{q^{k(k+r-c)}(1-q^{k+r})}{(q; q)_k(q; q)_{k+r}},\\
q^rL_{c-1}(r)&=\sum_{k\geq \max(0, -r)}\frac{q^{k(k+r-c)}q^{k+r}}{(q; q)_k(q; q)_{k+r}},
\end{align*}
where the second equality in the first line follows because the numerator vanishes at the possible endpoint $k=-r$. Adding the two lines gives
\begin{align*}
L_c(r)=L_{c-1}(r-1)+q^rL_{c-1}(r).
\end{align*}
Let the right side of \eqref{lem1} be $R_c(r)$.
Using the convention that the Gaussian binomial coefficient vanishes outside its standard range $0\leq a\leq c$, the $q$-Pascal identity
\begin{align*}
\begin{bmatrix}c\\a\end{bmatrix}_q=\begin{bmatrix}c-1\\a\end{bmatrix}_q+q^{c-a}\begin{bmatrix}c-1\\a-1\end{bmatrix}_q
\end{align*}
shows that $R_c(r)$ satisfies the same recurrence
\begin{align*}
R_c(r)=R_{c-1}(r-1)+q^rR_{c-1}(r).
\end{align*}
Since $R_0(r)=1/J_1=L_0(r)$, induction on $c$ proves the result.
\end{proof}

In 2016, Calinescu, Milas and Penn~\cite{CMP16} considered a family of Nahm sums associated with the tadpole Cartan matrix $T_r=(a_{ij})_{r\times r}$ with entries defined as 
\begin{align*}
a_{ij}=\begin{cases}
1,&\text{ if }i=j=r,\\
2, &\text{ if }i=j< r,\\
-1, & \text{ if }|i-j|=1,\\
0, &\text{ otherwise}. 
\end{cases}
\end{align*}
Let 
\begin{align}
\chi_r(x_1, ..., x_r)=\chi_r(x_1, ..., x_r; q):=\sum_{n=(n_1, ..., n_r)^T\in (\bZ_{\geq 0})^r}\frac{q^{\frac{1}{2}n^TT_rn}x_1^{n_1}\cdots x_{r}^{n_r}}{(q; q)_{n_1}\cdots (q; q)_{n_r}}
\end{align}
be a generalized tadpole Nahm sum. Our next result requires the case $r=4$. The rank four generalized tadpole Nahm sum is then given by
\begin{align}
\chi_4(x_1, x_2, x_3, x_4; q):=\sum_{n_1, n_2, n_3, n_4\geq 0}\frac{x_1^{n_1}x_2^{n_2}x_3^{n_3}x_4^{n_4}q^{n_1^2+n_2^2+n_3^2+n_4^2/2-n_1n_2-n_2n_3-n_3n_4}}{(q; q)_{n_1}(q; q)_{n_2}(q; q)_{n_3}(q; q)_{n_4}}.\label{rankfour}
\end{align}
We have discovered the following rank reduction formula, which will serve as the starting point for the subsequent proofs.
\begin{lemma}\label{lem:2}
For arbitrary nonzero $u$, $w$ and $t$, we have
\begin{align}\label{lem2}
\chi_4(u, u^{-1}, w, t; q) = \frac{1}{J_1}
\sum_{r \in \mathbb{Z}} \sum_{a, b, d \geq 0}
\frac{u^r w^{a+b} t^d}{(q; q)_a (q; q)_b (q; q)_d}
\, q^{r^2 + ar + a^2 + ab + b^2 + d^2 / 2 - (a+b)d}.
\end{align}
\end{lemma}

\begin{proof}
For $r\in\bZ$, we set
\begin{align*}
(n_1, n_2, n_3, n_4)= (k+r, k, c, d).
\end{align*}
Then, by first summing over $k\geq \max(0, -r)$, we have
\begin{align*}
\chi_4(u, u^{-1}, w, t; q)&=\sum_{r\in \bZ}\sum_{c, d\geq 0}\frac{u^rw^ct^dq^{r^2+c^2-cd+d^2/2}}{(q; q)_c(q; q)_d}\sum_{k\geq \max(0, -r)}\frac{q^{k(k+r-c)}}{(q; q)_k(q; q)_{k+r}}\\
&=\frac{1}{J_1}\sum_{r\in \bZ}\sum_{c, d\geq 0}\frac{u^rw^ct^dq^{r^2+c^2-cd+d^2/2}}{(q; q)_c(q; q)_d}\sum_{a=0}^cq^{a(r-c+a)}\begin{bmatrix}c\\a\end{bmatrix}_q\quad\text{(by Lemma \ref{lem:1})}\\
&=\frac{1}{J_1}\sum_{r\in \bZ}\sum_{a, b, d\geq 0}\frac{u^rw^{a+b}t^dq^{r^2+ar+a^2+ab+b^2+d^2/2-(a+b)d}}{(q; q)_d(q; q)_{a+b}}\begin{bmatrix}a+b\\a\end{bmatrix}_q\quad\text{(using $c=a+b$)}\\
&=\frac{1}{J_1}
\sum_{r \in \mathbb{Z}} \sum_{a, b, d \geq 0}
\frac{u^r w^{a+b} t^d}{(q; q)_a (q; q)_b (q; q)_d}
\, q^{r^2 + ar + a^2 + ab + b^2 + d^2 / 2 - (a+b)d},
\end{align*}
which is precisely the desired identity \eqref{lem2}.
\end{proof}

The next result is the celebrated Lebesgue identity, but note that the form we will use differs from the classical one by a shift of the parameter.

\begin{lemma}[Lebesgue identity]\label{lem:3}
We have
\begin{align}
\sum_{n\geq 0}\frac{(-z; q)_nq^{n(n+1)/2}}{(q; q)_n}=(-q; q)_{\infty}(-zq; q^2)_{\infty}.
\end{align}
\end{lemma}
For a base $q$ with $|q|<1$, define
\begin{align}
{}_r\phi_s \left[
\begin{matrix}
a_1, \ldots, a_r \\
b_1, \ldots, b_s
\end{matrix}
; q, z
\right]
:= \sum_{n \geq 0}
\frac{(a_1, \ldots, a_r; q)_n}{(q, b_1, \ldots, b_s; q)_n}
\left( (-1)^n q^{\binom{n}{2}} \right)^{1+s-r} z^n.
\end{align}
Recall that the following two classical summations \cite[Appendix II.10-II.11]{GR90}
\begin{align}
{}_2\phi_2 \left[ \begin{matrix} a, q/a \\ -q, b \end{matrix} ; q, -b \right]
&= \frac{(ab, bq/a; q^2)_\infty}{(b; q)_\infty},\label{II10}\\
{}_2\phi_2 \left[ \begin{matrix} a^2, b^2 \\ abq^{1/2}, -abq^{1/2} \end{matrix} ; q, -q \right]
&= \frac{(a^2q, b^2q; q^2)_\infty}{(q, a^2b^2q; q^2)_\infty},\label{II11}
\end{align}
and the limiting $q$-Gauss sum
\begin{align}
\sum_{n \geq 0}
\frac{(a; q)_n}{(q, c; q)_n}
(-1)^n q^{\binom{n}{2}}
\left(\frac{c}{a}\right)^n
= \frac{(c/a; q)_\infty}{(c; q)_\infty}.\label{II8}
\end{align}
The identity \eqref{II8} follows from the ordinary $q$-Gauss summation \cite[Appendix II.8]{GR90} by letting one upper parameter tend to infinity. With these results in hand, we obtain the following six identities that will be used subsequently. 

\begin{lemma}\label{lem:4}
We have
\begin{align}
\sum_{m \geq 0} \frac{q^{m^2}(-1; q^2)_m}{(q; q)_{2m}}
&= \frac{(-q; q^2)_\infty}{(q; q^2)_\infty}, \label{lem41} \\
\sum_{m \geq 0} \frac{q^{m(m+1)}(-q; q^2)_m}{(q; q)_{2m+1}}
&= \frac{(-q; q^2)_\infty}{(q^2; q^4)_\infty^2}, \label{lem42} \\
\sum_{m \geq 0} \frac{q^{m^2 - m}(-q; q^2)_m}{(q; q)_{2m}}
&= 2 \frac{(-q^2; q^2)_\infty}{(q; q^2)_\infty}, \label{lem43} \\
\sum_{m \geq 0} \frac{q^{m^2}(-q^2; q^2)_m}{(q; q)_{2m+1}}
&= \frac{(-q; q^2)_\infty}{(q; q^2)_\infty}, \label{lem44} \\
\sum_{m \geq 0} \frac{q^{2m^2-m}(-q; q^4)_m}{(q^2; q^2)_{2m}}
&= \frac{(-q; q^4)_\infty}{(q^2; q^4)_\infty}, \label{lem45} \\
\sum_{m \geq 0} \frac{q^{2m^2+m}(-q^3; q^4)_m}{(q^2; q^2)_{2m+1}}
&= \frac{(-q^3; q^4)_\infty}{(q^2; q^4)_\infty}. \label{lem46}
\end{align}
\end{lemma}

\begin{proof}
For \eqref{lem41}, with $(q, a, b)\ri (q^2, -1, q)$, the cancellation in \eqref{II10} gives 
\begin{align*}
\sum_{m \geq 0} \frac{q^{m^2} (-1; q^2)_m}{(q; q)_{2m}}
= {}_2\phi_2 \left[
\begin{matrix}
-1, -q^2 \\
-q^2, q
\end{matrix}
; q^2, -q
\right]
= \frac{(-q, -q^3; q^4)_\infty}{(q; q^2)_\infty}
= \frac{(-q; q^2)_\infty}{(q; q^2)_\infty}.
\end{align*}
For \eqref{lem42}, with $(q, a^2, b^2)\ri (q^2, -q, -q^3)$, the cancellation in \eqref{II11} gives
\begin{align*}
\sum_{m \geq 0} \frac{q^{m(m+1)}(-q; q^2)_m}{(q; q)_{2m+1}}
= \frac{1}{1-q} \, {}_2\phi_2 \left[
\begin{matrix}
-q, -q^3 \\
q^3, -q^3
\end{matrix}
; q^2, -q^2
\right]
= \frac{1}{1-q} \frac{(-q^3, -q^5; q^4)_\infty}{(q^2, q^6; q^4)_\infty}
= \frac{(-q; q^2)_\infty}{(q^2; q^4)_\infty^2}.
\end{align*}
For \eqref{lem43}, taking $(q, a, c)\ri (q^2, -q, q)$ in \eqref{II8} gives
\begin{align*}
\sum_{m \geq 0} \frac{q^{m^2 - m}(-q; q^2)_m}{(q^2; q^2)_m (q; q^2)_m}
= \frac{(-1; q^2)_\infty}{(q; q^2)_\infty}
= 2 \frac{(-q^2; q^2)_\infty}{(q; q^2)_\infty}.
\end{align*}
For \eqref{lem44}, taking $(q, a, c)\ri (q^2, -q^2, -q^3)$ in \eqref{II8} gives
\begin{align*}
\sum_{m \geq 0} \frac{q^{m^2} (-q^2; q^2)_m}{(q; q)_{2m+1}}
= \frac{1}{1-q} \frac{(-q; q^2)_\infty}{(q^3; q^2)_\infty}
= \frac{(-q; q^2)_\infty}{(q; q^2)_\infty}.
\end{align*}
For \eqref{lem45}, taking $(q, a, c)\ri (q^4, -q, q^2)$ in \eqref{II8} gives
\begin{align*}
\sum_{m \geq 0} \frac{q^{2m^2-m}(-q; q^4)_m}{(q^4; q^4)_m (q^2; q^4)_m}
= \frac{(-q; q^4)_\infty}{(q^2; q^4)_\infty}.
\end{align*}
For \eqref{lem46}, taking $(q, a, c)\ri (q^4, -q^3, q^6)$ in \eqref{II8} gives
 \begin{align*}
\sum_{m \geq 0} \frac{q^{2m^2+m}(-q^3; q^4)_m}{(q^2; q^2)_{2m+1}}
= \frac{1}{1-q^2} \frac{(-q^3; q^4)_\infty}{(q^6; q^4)_\infty}
= \frac{(-q^3; q^4)_\infty}{(q^2; q^4)_\infty}.
\end{align*}
This completes the proof.
\end{proof}

\section{Proof of Theorem \ref{thm:1}}\label{sec:11}

In this section, we provide analytic proofs of the three identities in Theorem \ref{thm:1} individually. Before proceeding with the proofs, we introduce some notation. Let
\begin{align*}
P:=(-q; q)_{\infty},\quad H:=(-q; q^2)_{\infty},\quad R:=(-q^2; q^2)_{\infty}.
\end{align*}
Then $P=HR=J_2/J_1$, $R=J_4/J_2$, and
\begin{align*}
(q; q^2)_{\infty}=\frac{J_1}{J_2},\quad (q^2; q^4)_{\infty}=\frac{J_2}{J_4}=R^{-1}.
\end{align*}
The Jacobi triple product \eqref{jacobi} gives 
\begin{align}
\theta_0=\theta_0(q):=\sum_{r\in \bZ}q^{r^2}=J_2H^2,\quad \theta_1=\theta_1(q):=\sum_{r\in \bZ}q^{r^2+r}=2J_2R^2.\label{ja}
\end{align}

\begin{proof}[Proof of \eqref{thm13}]
For \eqref{thm13}, using the substitution $(n_1, n_2, n_3, n_4)=(k, j, l, i)$ in \eqref{rankfour}, we have
\begin{align*}
\text{Left side of }\eqref{thm13}&=\chi_4(q, q^{-1}, 1, q^{\frac{1}{2}}; q)\\
&=\frac{1}{J_1}
\sum_{r \in \mathbb{Z}} \sum_{a, b, d \geq 0}
\frac{q^{r^2 + ar + a^2 + ab + b^2 + d^2 / 2 - (a+b)d+r+d/2}}{(q; q)_a (q; q)_b (q; q)_d},
\end{align*}
where the second equality follows from Lemma \ref{lem:2}. Summing over $d\geq 0$ by \eqref{II2} and using
\begin{align*}
(-q^{1-a-b}; q)_{\infty}=Pq^{-\binom{a+b}{2}}(-1; q)_{a+b}
\end{align*}
yields
\begin{align}
\text{Left side of }\eqref{thm13}&=\frac{1}{J_1}\sum_{r\in \bZ}\sum_{a, b\geq 0}\frac{q^{r^2+ar+a^2+ab+b^2+r}}{(q; q)_a(q; q)_b}\sum_{d\geq 0}\frac{q^{d^2/2+d/2-(a+b)d}}{(q; q)_d}\nonumber\\
&=\frac{1}{J_1}\sum_{r\in \bZ}\sum_{a, b\geq 0}\frac{q^{r^2+ar+a^2+ab+b^2+r}}{(q; q)_a(q; q)_b}(-q^{1-a-b}; q)_{\infty}\quad\text{(by \eqref{II2})}\nonumber\\
&=\frac{P}{J_1}\sum_{r\in \bZ}\sum_{a, b\geq 0}\frac{q^{r^2+ar+a^2+ab+b^2+r-\binom{a+b}{2}}(-1; q)_{a+b}}{(q; q)_a(q; q)_b}\nonumber\\
&=\frac{P}{J_1}U_3(q),\label{13_1}
\end{align}
where
\begin{align*}
U_3(q):=\sum_{r\in \bZ}\sum_{a, b\geq 0}\frac{(-1; q)_{a+b}q^{r^2+(a+1)r+\binom{a+1}{2}+\binom{b+1}{2}}}{(q; q)_a(q; q)_b}.
\end{align*}
Next we split $a$ into its even and odd parts. For $a=2m$ we have
\begin{align*}
r^2+(2m+1)r=(r+m)^2+(r+m)-m(m+1),
\end{align*}
whereas for $a=2m+1$ we have
\begin{align*}
r^2+(2m+2)r=(r+m+1)^2-(m+1)^2.
\end{align*}
Hence, after these operations, summing over $b\geq 0$ gives
\begin{align}
U_3(q)=&\sum_{r\in \bZ}\sum_{m\geq 0}\frac{(-1; q)_{2m}q^{(r+m)^2+(r+m)+m^2}}{(q; q)_{2m}}\sum_{b\geq 0}\frac{(-q^{2m}; q)_bq^{\binom{b+1}{2}}}{(q; q)_b}\nonumber\\
&+\sum_{r\in \bZ}\sum_{m\geq 0}\frac{(-1; q)_{2m+1}q^{(r+m+1)^2+m^2+m}}{(q; q)_{2m+1}}\sum_{b\geq 0}\frac{(-q^{2m+1}; q)_bq^{\binom{b+1}{2}}}{(q; q)_b}\nonumber\\
=&P\sum_{r\in \bZ}\sum_{m\geq 0}\frac{(-1; q)_{2m}q^{(r+m)^2+(r+m)+m^2}}{(q; q)_{2m}}(-q^{2m+1}; q^2)_{\infty}\nonumber\\
&+P\sum_{r\in \bZ}\sum_{m\geq 0}\frac{(-1; q)_{2m+1}q^{(r+m+1)^2+m^2+m}}{(q; q)_{2m+1}}(-q^{2m+2}; q^2)_{\infty}\quad\text{(by Lemma \ref{lem:3})}\nonumber\\
=&PH\sum_{m\geq 0}\frac{(-1; q^2)_mq^{m^2}}{(q; q)_{2m}}\sum_{r\in\bZ}q^{(r+m)^2+(r+m)}+2PR\sum_{m\geq 0}\frac{(-q; q^2)_mq^{m^2+m}}{(q; q)_{2m+1}}\sum_{r\in \bZ}q^{(r+m+1)^2}\nonumber\\
=&\theta_1PH\sum_{m\geq 0}\frac{q^{m^2}(-1; q^2)_m}{(q; q)_{2m}}+2\theta_0PR\sum_{m\geq 0}\frac{q^{m(m+1)}(-q; q^2)_m}{(q; q)_{2m+1}}.\quad\text{(by \eqref{ja})}\label{u3}
\end{align}
By \eqref{lem41}, \eqref{lem42} and \eqref{ja}, we know that each part in \eqref{u3} equals $2J_1P^5$. Thus we have $U_3(q)=4J_1P^5$, and by \eqref{13_1}
\begin{align*}
\text{Left side of }\eqref{thm13}=4P^6=4\frac{J_2^6}{J_1^6}.
\end{align*}
This completes the proof.
\end{proof}
The proofs of the remaining two identities in Theorem \ref{thm:1} are similar to the proof of \eqref{thm13}, and we will omit some details and leave the remaining details to the interested reader.

\begin{proof}[Proof of \eqref{thm11}]
For \eqref{thm11}, using the substitution $(n_1, n_2, n_3, n_4)=(k, j, l, i)$ in \eqref{rankfour}, we have
\begin{align*}
\text{Left side of }\eqref{thm11}&=\chi_4(q^2, q^{-2}, q, q^{-\frac{1}{2}}; q)\\
&=\frac{1}{J_1}
\sum_{r \in \mathbb{Z}} \sum_{a, b, d \geq 0}
\frac{q^{r^2 + ar + a^2 + ab + b^2 + d^2 / 2 - (a+b)d+2r+a+b-d/2}}{(q; q)_a (q; q)_b (q; q)_d},
\end{align*}
where the second equality follows from Lemma \ref{lem:2}. Summing over $d\geq 0$ by \eqref{II2} and using
\begin{align*}
(-q^{-a-b}; q)_{\infty}=2Pq^{-\binom{a+b+1}{2}}(-q; q)_{a+b}
\end{align*}
yields
\begin{align}
\text{Left side of \eqref{thm11}}=\frac{2P}{J_1}U_1(q),\label{11_2}
\end{align}
where
\begin{align*}
U_1(q):=\sum_{r\in \bZ}\sum_{a, b\geq 0}\frac{(-q; q)_{a+b}q^{r^2+(a+2)r+\binom{a+1}{2}+\binom{b+1}{2}}}{(q; q)_a(q; q)_b}.
\end{align*}
For $a=2m$ and $a=2m+1$, respectively, we have
\begin{align*}
r^2+(2m+2)r&=(r+m+1)^2-(m+1)^2,\\
r^2+(2m+3)r&=(r+m+1)^2+(r+m+1)-(m+1)(m+2).
\end{align*}
Using $(-q; q)_{2m}=(-q; q^2)_{m}(-q^2; q^2)_m$ and $(-q; q)_{2m+1}=(-q; q^2)_{m+1}(-q^2; q^2)_m$ and summing over $b\geq 0$ by Lemma \ref{lem:3}, the parity split yields
\begin{align}
U_1(q)=q^{-1}\theta_0PR\sum_{m\geq 0}\frac{q^{m^2-m}(-q; q^2)_m}{(q; q)_{2m}}+q^{-1}\theta_1PH\sum_{m\geq 0}\frac{q^{m^2}(-q^2; q^2)_m}{(q; q)_{2m+1}}.\label{11_1}
\end{align}
By \eqref{lem43}, \eqref{lem44} and \eqref{ja}, we know that each part in \eqref{11_1} equals $2q^{-1}J_1P^5$. Therefore, we have $U_1(q)=4q^{-1}J_1P^5$, and \eqref{11_2} gives 
\begin{align*}
\text{Left side of \eqref{thm11}}=8q^{-1}P^6=8q^{-1}\frac{J_2^6}{J_1^6}.
\end{align*}
This completes the proof.
\end{proof}

\begin{proof}[Proof of \eqref{thm12}]
For \eqref{thm12}, using the substitution $(n_1, n_2, n_3, n_4)=(k, j, l, i)$ with $q\ri q^2$ in \eqref{rankfour}, we have
\begin{align*}
\text{Left side of }\eqref{thm12}&=\chi_4(q^3, q^{-3}, q, 1; q^2)\\
&=\frac{1}{J_2}
\sum_{r \in \mathbb{Z}} \sum_{a, b, d \geq 0}
\frac{q^{2r^2 + 2ar + 2a^2 + 2ab + 2b^2 + d^2 - 2(a+b)d+3r+a+b}}{(q^2; q^2)_a (q^2; q^2)_b (q^2; q^2)_d},
\end{align*}
where the second equality follows from Lemma \ref{lem:2}. Summing over $d\geq 0$ by \eqref{II2} and using
\begin{align*}
(-q^{1-2a-2b}; q)_{\infty}=Hq^{-(a+b)^2}(-q; q^2)_{a+b}
\end{align*}
yields
\begin{align}
\text{Left side of \eqref{thm12}}=\frac{H}{J_2}U_2(q),\label{12_1}
\end{align}
where
\begin{align*}
U_2(q):=\sum_{r\in \bZ}\sum_{a, b\geq 0}\frac{(-q; q^2)_{a+b}q^{2r^2+(2a+3)r+a(a+1)+b(b+1)}}{(q^2; q^2)_a(q^2; q^2)_b}.
\end{align*}
Let 
\begin{align}
\theta_2:=\sum_{r\in \bZ}q^{2r^2+r}=J_4H,\quad\theta_3:=\sum_{r\in\bZ}q^{2r^2+3r}=q^{-1}\theta_2.\label{12_3}
\end{align}
For $a=2m$ and $a=2m+1$, respectively, we have
\begin{align*}
2r^2+(4m+3)r&=2(r+m)^2+3(r+m)-2m(m+3),\\
2r^2+(4m+5)r&=2(r+m+1)^2+(r+m+1)-(m+1)(2m+3).
\end{align*}
After splitting $(-q; q^2)_{2m}$ and $(-q; q^2)_{2m+1}$ into residue classes modulo $4$ and summing over $b\geq 0$ by Lemma \ref{lem:3}, the parity decomposition yields
\begin{align}
U_2(q)=q^{-1}\theta_2R(-q^3; q^4)_{\infty}\sum_{m\geq 0}\frac{q^{2m^2-m}(-q; q^4)_m}{(q^2; q^2)_m}+q^{-1}\theta_2R(-q; q^4)_{\infty}\sum_{m\geq 0}\frac{q^{2m^2+m}(-q^3; q^4)_{m}}{(q^2; q^2)_{2m+1}}.\label{12_2}
\end{align}
By \eqref{lem45}, \eqref{lem46} and \eqref{12_3}, we know that each part in \eqref{12_2} equals $q^{-1}J_4P^2$. Hence, we have $U_2(q)=2q^{-1}J_4P^2$, and \eqref{12_1} gives 
\begin{align*}
\text{Left side of \eqref{thm12}}=2q^{-1}HJ_4P^2/J_2=2q^{-1}P^3=2q^{-1}\frac{J_2^3}{J_1^3},
\end{align*}
since $HJ_4/J_2=HR=P$. This completes the proof.
\end{proof}

\section{Proof of Theorem \ref{thm:2}}\label{sec:12}

In this section, we provide analytic proofs of the two identities in Theorem \ref{thm:2}. Prior to presenting this proof, we require three further tools.

\begin{lemma}[{cf.~\cite[Equations (2.6.12) and (2.6.14)]{LSZ08}}]\label{lem:sla}
We have
\begin{align}
\sum_{s\geq 0}\frac{q^{s(s+1)}(-q^2; q^2)_s}{(q; q)_{2s+1}}&=\frac{f(q, q^5)}{f(-q^2, -q^2)},\label{sla1}\\
\sum_{s\geq 0}\frac{q^{s^2}(-q; q^2)_{s}}{(q; q)_{2s}}&=\frac{f(q^2, q^4)}{f(-q, -q^3)}.\label{sla2}
\end{align}
\end{lemma}

\begin{lemma}[{cf. \cite[Equation (1.13) in Theorem~1.3]{SW25}}]\label{lem:SW}
We have
\begin{align}
\sum_{a, b, d\geq 0}\frac{t^{4a^2+4ab+3b^2-2bd+d^2-2b+2d}}{(t^4; t^4)_a(t^4; t^4)_b(t^4; t^4)_d}=\frac{(t^{12}; t^{12})_{\infty}}{(t^8; t^8)_{\infty}(t, t^{11}; t^{12})_{\infty}(t^5, t^7; t^{12})_{\infty}}.\label{SW}
\end{align}
\end{lemma}

\begin{lemma}[{cf. \cite[Proposition 2.2 (ii)]{BBG94}}]\label{lem:BBG}
Let 
\begin{align}
c(q):=\sum_{n, m\in \bZ}q^{(n+\frac{1}{3})^2+(n+\frac{1}{3})(m+\frac{1}{3})+(m+\frac{1}{2})^2},
\end{align}
then we have
\begin{align}
c(q)=3q^{\frac{1}{3}}\frac{(q^3; q^3)_{\infty}^3}{(q; q)_{\infty}}.\label{cubic}
\end{align}
\end{lemma}

\begin{proof}[Proof of \eqref{thm2}]
For the left side of \eqref{thm2}, we first sum over $i\geq 0$ to get 
\begin{align}
\text{Left side of \eqref{thm2}}&=\sum_{j,k,l \geq 0}
\frac{
  q^{ 2j^2 + 2k^2 + 2l^2 - 2jk - 2jl - 2j + 2k}
}{
  (q^2; q^2)_j (q^2; q^2)_k (q^2; q^2)_l
}\sum_{i\geq 0}\frac{q^{i^2+2i(k-j+1)}}{(q^2; q^2)_i}\nonumber\\
&=\sum_{j,k,l \geq 0}
\frac{
  q^{ 2j^2 + 2k^2 + 2l^2 - 2jk - 2jl - 2j + 2k}
}{
  (q^2; q^2)_j (q^2; q^2)_k (q^2; q^2)_l
}(-q^{2k-2j+3}; q^2)_{\infty}\quad\text{(by \eqref{II2})}\nonumber\\
&=H\sum_{j,k,l \geq 0}
\frac{
  q^{ 2j^2 + 2k^2 + 2l^2 - 2jk - 2jl - 2j + 2k}
}{
  (q^2; q^2)_j (q^2; q^2)_k (q^2; q^2)_l(-q; q^2)_{k-j+1}
}.\label{2_1}
\end{align}
Next we set $r=k-j$, then the exponent of $q$ in numerator of \eqref{2_1} is
\begin{align*}
2j^2 + 2k^2 + 2l^2 - 2jk - 2jl - 2j + 2k&=2j^2+2(r+j)^2+2l^2-2j(r+j)-2jl-2j+2(r+j)\\
&=2(j(j+r-l)+r^2+r+l^2).
\end{align*}
And applying Lemma \ref{lem:1} with $q\ri q^2$ to \eqref{2_1}
gives
\begin{align}
&\text{Left side of }\eqref{thm2}\nonumber\\=&H\sum_{r\in\bZ}\sum_{l\geq 0}\frac{q^{2r^2+2r+2l^2}}{(q^2; q^2)_l(-q; q^2)_{r+1}}\sum_{j\geq \max(0, -r)}\frac{q^{2j(j+r-l)}}{(q^2; q^2)_j(q^2; q^2)_{r+j}}\quad\text{(by $r=k-j$ in \eqref{2_1})}\nonumber\\
=&\frac{H}{J_2}\sum_{r\in \bZ}\sum_{l\geq 0}\frac{q^{2r^2+2r+2l^2}}{(q^2; q^2)_l(-q; q^2)_{r+1}}\sum_{b=0}^{l}q^{2b(r-l+b)}\begin{bmatrix}l\\b\end{bmatrix}_{q^2}\quad\text{(by Lemma \ref{lem:1})}\nonumber\\
=&\frac{H}{J_2}\sum_{r\in \bZ}\sum_{a, b\geq 0}\frac{q^{2r^2+2r+2(a+b)^2+2br-2b(a+b)+2b^2}}{(-q; q^2)_{r+1}(q^2; q^2)_a(q^2; q^2)_b}\quad\text{(by $l=a+b$)}\nonumber\\
=&\frac{H}{J_2}T(q),\quad\text{(by $r+1=c$)}\label{2_2}
\end{align}
where
\begin{align}
T(q):=\sum_{c\in \bZ}\sum_{a, b\geq 0}\frac{q^{2a^2+2ab+2b^2+2bc+2c^2-2b-2c}}{(q^2; q^2)_{a}(q^2; q^2)_b(-q; q^2)_{c}}.
\end{align}
For every $c\in \bZ$, by \eqref{II2} we have
\begin{align*}
\frac{H}{(-q; q^2)_c}=(-q^{2c+1}; q^2)_{\infty}=\sum_{d\geq 0}\frac{q^{d^2+2cd}}{(q^2; q^2)_d}.
\end{align*}
Therefore, we obtain 
\begin{align}
H\cdot T(q)=\sum_{c\in \bZ}\sum_{a, b, d\geq 0}\frac{q^{2a^2+2ab+2b^2-2b+d^2+2c^2+2c(b+d-1)}}{(q^2; q^2)_a(q^2; q^2)_b(q^2; q^2)_d}.\label{2_3}
\end{align}
Now we set $q^2=t^4$, that is, $t=q^{1/2}$. Then the exponent identity
\begin{align}
4a^2&+4ab+4b^2-4b+2d^2+4c^2+4c(b+d-1)\nonumber\\
&=4a^2+4ab+3b^2-2bd+d^2-2b+2d+(2c+b+d-1)^2-1\label{ex}
\end{align}
separates the bilateral theta sum from the nonnegative triple sum in \eqref{2_3}, that is,
\begin{align}
H\cdot T(t^2)=\sum_{a, b, d\geq 0}\frac{t^{4a^2+4ab+3b^2-2bd+d^2-2b+2d}}{(t^4; t^4)_a(t^4; t^4)_b(t^4; t^4)_d}\sum_{c\in \bZ}t^{(2c+b+d-1)^2-1}.\label{2_4}
\end{align}
We can define
\begin{align}
Z(t):=\sum_{a, b, d\geq 0}\frac{t^{4a^2+4ab+3b^2-2bd+d^2-2b+2d}}{(t^4; t^4)_a(t^4; t^4)_b(t^4; t^4)_d},
\end{align}
and let $Z_e(t)$ and $Z_o(t)$ be its portions with $b+d$ even and odd, respectively. Since
\begin{align*}
\sum_{n\in \bZ}t^{(2n)^2}=f(t^4, t^4),\quad \sum_{n\in \bZ}t^{(2n+1)^2}=tf(1, t^8),
\end{align*}
then \eqref{2_4} gives
\begin{align}
H\cdot T(t^2)=t^{-1}(f(t^4, t^4)Z_o(t)+tf(1, t^8)Z_e(t)).\label{thm2-2}
\end{align}
Moreover, using \eqref{SW} we have 
\begin{align}
Z(t)=\frac{(t^{12}; t^{12})_{\infty}}{(t^8; t^8)_{\infty}(t, t^{11}; t^{12})_{\infty}(t^5, t^7; t^{12})_{\infty}}.
\end{align}
By Jacobi triple product identity \eqref{jacobi}, this product is equivalently
\begin{align}
Z(t)=\frac{(-t^2; t^4)_{\infty}}{(t^4; t^4)_{\infty}}f(t, t^5).\label{thm2-1}
\end{align}
The factor $(-t^2; t^4)_{\infty}/(t^4; t^4)_{\infty}$ in \eqref{thm2-1} contains only even powers of $t$. The even-odd dissection
\begin{align}
f(t, t^5)=f(t^8, t^{16})+tf(t^4, t^{20})
\end{align}
therefore gives 
\begin{align*}
Z_e(t)=\frac{(-t^2; t^4)_{\infty}}{(t^4; t^4)_{\infty}}f(t^8, t^{16}),\quad Z_o(t)=t\frac{(-t^2; t^4)_{\infty}}{(t^4; t^4)_{\infty}}f(t^4, t^{20}).
\end{align*}
Substituting $Z_e(t)$ and $Z_o(t)$ into \eqref{thm2-2} yields
\begin{align}
H\cdot T(t^2)=\frac{(-t^2; t^4)_{\infty}}{(t^4; t^4)_{\infty}}(f(t^4, t^4)f(t^4, t^{20})+f(1, t^8)f(t^8, t^{16})).\label{thm2-3}
\end{align}
Now we claim that
\begin{align}
f(t^4, t^4)f(t^4, t^{20})+f(1, t^8)f(t^8, t^{16})=\sum_{m, n\in \bZ}t^{4m^2+4mn+4n^2+4m+4n}.\label{2_31}
\end{align}
Indeed, for the change of variables $u=m+n$ and $v=m-n$, the integers $u$, $v$ have the same parity and 
\begin{align*}
4(m^2+mn+n^2+m+n)=3u^2+v^2+4u.
\end{align*}
If $u=2r$ and $v=2s$, the exponent is $12r^2+8r+4s^2$; replacing $r$ by $-r$ gives the product $f(t^4, t^4)f(t^4, t^{20})$ by \eqref{theta} and \eqref{jacobi}. If $u=2r+1$ and $v=2s+1$, set $R=r+1$ and $S=-s$; the exponent becomes $12R^2-4R+4S(S-1)$, giving the product $f(1, t^8)f(t^8, t^{16})$ by \eqref{theta} and \eqref{jacobi}. 

Then the shifted cubic theta identity is
\begin{align}
\sum_{m, n\in \bZ}t^{4m^2+4mn+4n^2+4m+4n}=\sum_{m, n\in \bZ}q^{2m^2+2mn+2n^2+2m+2n}=3\frac{(q^6; q^6)_{\infty}^3}{(q^2; q^2)_{\infty}}=3\frac{J_6^3}{J_2},\label{thm2-4}
\end{align}
by letting $q\ri q^2$ in \eqref{cubic}. Since 
\begin{align*}
\frac{(-t^2; t^4)_{\infty}}{(t^4; t^4)_{\infty}}=\frac{(-q; q^2)_{\infty}}{(q^2; q^2)_{\infty}}=\frac{H}{J_2},
\end{align*}
then \eqref{thm2-3} and \eqref{thm2-4} imply
\begin{align*}
T(q)=3\frac{J_6^3}{J_2^2}.
\end{align*}
Finally, by \eqref{2_2} we have
\begin{align*}
\text{Left side of }\eqref{thm2}&=3H\frac{J_6^3}{J_2^3}=3\frac{J_2^2}{J_1J_4}\frac{J_6^3}{J_2^3}=3\frac{J_6^3}{J_1J_2J_4}.
\end{align*}
This completes the proof.
\end{proof}

\begin{proof}[Proof of \eqref{thm21}]
For the left side of \eqref{thm21}, we have
\begin{align}
&\text{Left side of \eqref{thm21}}\nonumber\\
=&\sum_{i, k\geq 0}\sum_{r\in \bZ}\frac{q^{\frac{N^2+N}{2}+\frac{k^2+k}{2}+r^2-r(N+1)}}{(q; q)_i(q; q)_k}\sum_{l\geq \max(0, -r)}\frac{q^{l(l+r-N)}}{(q; q)_{l}(q; q)_{r+l}}\quad\text{(by $r=j-l$ and $N=i+k$)}\nonumber\\
=&\frac{1}{J_1}\sum_{i, k\geq 0}\sum_{r\in \bZ}\frac{q^{\frac{N^2+N}{2}+\frac{k^2+k}{2}+r^2-r(N+1)}}{(q; q)_i(q; q)_k}\sum_{a=0}^{N}q^{a(r-N+a)}\begin{bmatrix}N\\a\end{bmatrix}_q\quad\text{(by Lemma \ref{lem:1})}\nonumber\\
=&\frac{1}{J_1}\sum_{i, k\geq 0}\frac{q^{\frac{N^2+N}{2}+\frac{k^2+k}{2}}}{(q; q)_i(q; q)_k}\sum_{h=0}^{N}q^{h^2-Nh}\begin{bmatrix}N\\h\end{bmatrix}_q\sum_{r\in \bZ}q^{r^2-(h+1)r}.\quad\text{(by $h=N-a$)}\label{4_1}
\end{align}
Next, summing over $i+k=N$ and applying \eqref{finite-q-bino}, we obtain
\begin{align}
&\text{Left side of \eqref{thm21}}\nonumber\\
=&\frac{1}{J_1}\sum_{N\geq 0}\sum_{k=0}^N\frac{q^{\frac{k^2+k}{2}}}{(q; q)_{N-k}(q; q)_k}q^{\frac{N^2+N}{2}}\sum_{h=0}^{N}q^{h^2-Nh}\begin{bmatrix}N\\h\end{bmatrix}_q\sum_{r\in \bZ}q^{r^2-(h+1)r}\quad\text{(by \eqref{4_1})}\nonumber\\
=&\frac{1}{J_1}\sum_{N\geq 0}\frac{(-q; q)_N q^{\frac{N^2+N}{2}}}{(q; q)_N}\sum_{h=0}^{N}q^{h^2-Nh}\begin{bmatrix}N\\h\end{bmatrix}_q\sum_{r\in \bZ}q^{r^2-(h+1)r}.\quad\text{(by \eqref{finite-q-bino})}\label{4_2}
\end{align}
We now split $h$ by parity. If $h=2s$, then
\begin{align}
\sum_{r\in \bZ}q^{r^2-(2s+1)r}=q^{-s(s+1)}\theta_1;
\end{align}
and if $h=2s+1$, then 
\begin{align}
\sum_{r\in \bZ}q^{r^2-(2s+2)r}=q^{-(s+1)^2}\theta_0,
\end{align}
where $\theta_0$ and $\theta_1$ are defined in \eqref{ja}.
In the first case we write $N=2s+n$ and in the second case we write $N=2s+1+n$. Then by \eqref{4_2} we have
\begin{align}
\text{Left side of \eqref{thm21}}=\frac{\theta_1}{J_1}\sum_{s, n\geq 0}\frac{(-q; q)_{2s+n}q^{\frac{n^2+n}{2}+s^2}}{(q; q)_{2s}(q; q)_n}+\frac{\theta_0}{J_1}\sum_{s, n\geq 0}\frac{(-q; q)_{2s+1+n}q^{\frac{n^2+n}{2}+s(s+1)}}{(q; q)_{2s+1}(q; q)_n},\label{4_3}
\end{align}
since the relevant exponent identities are
\begin{align*}
\frac{N^2+N}{2}+h^2-Nh-s(s+1)&=\frac{n^2+n}{2}+s^2,\\
\frac{N^2+N}{2}+h^2-Nh-(s+1)^2&=\frac{n^2+n}{2}+s(s+1).
\end{align*}
Moreover, summing over $n$ in \eqref{4_3} and applying Lemma \ref{lem:3}, we have
\begin{align}
\text{Left side of \eqref{thm21}}=&\frac{PR\theta_1}{J_1}\sum_{s\geq 0}\frac{q^{s^2}(-q; q^2)_s}{(q; q)_{2s}}+\frac{PH\theta_0}{J_1}\sum_{s\geq 0}\frac{q^{s(s+1)}(-q^2; q^2)_s}{(q; q)_{2s+1}}\nonumber\\
=&\frac{PR\theta_1}{J_1}\frac{f(q^2, q^4)}{f(-q, -q^3)}+\frac{PH\theta_0}{J_1}\frac{f(q, q^5)}{f(-q^2, -q^2)}\quad\text{(by Lemma \ref{lem:sla})}\nonumber\\
=&\frac{P}{J_1^2}(f(1, q^2)f(q^2, q^4)+f(q, q)f(q, q^5)).\quad\text{(by \eqref{jacobi})}\nonumber\\
=&3\frac{J_2J_3^3}{J_1^4}.\quad\text{(by \eqref{2_31} and \eqref{thm2-4})}
\end{align}
This completes the proof.
\end{proof}

\section{Concluding remarks}\label{sec:conclusion}

In this paper, we give proofs of five conjectural identities on modular rank four Nahm sums proposed by Cao and Wang~\cite{CW25}. However, there is another conjectural identity in \cite{CW25} that we are currently unable to resolve:
\begin{conj}[{cf.~\cite[Conjecture 3.3]{CW25}}]\label{conj:CW33}
We have
\begin{align}
\sum_{i,j,k,l \geq 0} 
\frac{
  q^{2i^2 + j^2 + 2k^2 + l^2 - 2ik - 2il - jl + 2i - j - 2k}
}{
  (q^2; q^2)_i (q^2; q^2)_j (q^2; q^2)_k (q^2; q^2)_l
}
= 9 \frac{J_3^3 J_6^2}{J_1 J_2^4}.\label{CW33}
\end{align}
\end{conj}
 In this section, we first give a decomposition of \eqref{CW33}, and then present a simplified version of \eqref{thm21}, which connects to a conjecture of Shi and Wang \cite{SW25}. Now we introduce a conjecture. 

\begin{conj}\label{conj:LI}
Define
\begin{align}
\hZ_e(q):&=\sum_{a, j, s\geq 0}\frac{q^{a^2-aj+j^2-j+2s^2+2s-2sj}}{(q^2; q^2)_a(q^2; q^2)_j(q^2; q^2)_{2s}},\\
\hZ_o(q):&=\sum_{a, j, s\geq 0}\frac{q^{a^2-aj+j^2-j+2s^2+4s+1-(2s+1)j}}{(q^2; q^2)_a(q^2; q^2)_j(q^2; q^2)_{2s+1}},
\end{align}
then we have
\begin{align}
\hZ_e(q)&=3\frac{J_3^3}{J_1J_2^2J_6}f(q^4, q^8),\\
\hZ_o(q)&=3\frac{J_3^3}{J_1J_2^2J_6}f(q^2, q^{10}).
\end{align}
\end{conj}
Next we will see that if Conjecture \ref{conj:LI} is established, then \eqref{CW33} follows immediately. 

\begin{proposition}
Conjecture \ref{conj:LI} implies \eqref{CW33}.
\end{proposition}

\begin{proof}
For left side of \eqref{CW33}, setting $i=k+r$ and applying Lemma \ref{lem:2} with $q\ri q^2$ and $c=l=a+b$, then after simplification we have
\begin{align}
\text{Left side of \eqref{CW33}}=\frac{1}{J_2}\sum_{r\in \bZ}\sum_{a, b, j\geq 0}\frac{q^{2r^2+2r-2br+a^2-aj+b^2-bj+j^2-j}}{(q^2; q^2)_a(q^2; q^2)_b(q^2; q^2)_j}.\label{6_1}
\end{align}
We split $b$ by parity. For $b=2s$ we have
\begin{align}
\sum_{r\in \bZ}q^{2r^2+2r-4sr}=q^{-2s(s-1)}\theta_1(q^2),\label{6_2}
\end{align}
and for $b=2s+1$ we have
\begin{align}
\sum_{r\in \bZ}q^{2r^2+2r-2(2s+1)r}=q^{-2s^2}\theta_0(q^2).\label{6_3}
\end{align}
Then \eqref{6_1}, \eqref{6_2} and \eqref{6_3} give the decomposition
\begin{align}
\text{Left side of \eqref{CW33}}&=\frac{\theta_1(q^2)}{J_2}\hZ_e(q)+\frac{\theta_0(q^2)}{J_2}\hZ_o(q)\nonumber\\
&=3\frac{J_3^3}{J_1J_2^3J_6}(f(1, q^4)f(q^4, q^8)+f(q^2, q^2)f(q^2, q^{10}))\nonumber\\
&=9\frac{J_3^3J_6^2}{J_1J_2^4}.\quad\text{(by \eqref{2_31} and \eqref{thm2-4})}
\end{align}
This completes the proof.
\end{proof}

Now we establish a connection between \eqref{thm21} and the conjecture stated in \cite{SW25} as follows.

\begin{conj}[{cf.~\cite[Conjecture 1.4]{SW25}}]\label{conj:SW14}
We have
\begin{align}\label{SW14}
\sum_{a, b, d \geq 0} 
\frac{
  t^{4a^2 + 4ab + 3b^2 - 2bd + d^2 + 4a + 2b}
}{
  (t^4; t^4)_{a}(t^4; t^4)_{b}(t^4; t^4)_{d}
}
= 
\frac{
  (t^6; t^6)_\infty (t^8; t^8)_\infty (t^2, t^{10}; t^{12})_\infty
}{
  (t^4; t^4)_\infty^2 (t, t^{11}; t^{12})_\infty (t^5, t^7; t^{12})_\infty
}.
\end{align}
\end{conj}

We simplify the left side of \eqref{thm21} by an argument analogous to the proof of \eqref{thm2}. We perform the following steps on the left side of \eqref{thm21} in order: (1) sum over $i\geq 0$ by \eqref{II2}; (2) put $c=k-j$; (3) apply Lemma \ref{lem:1}; (4) set $l=a+b$. Then we obtain 
\begin{align}
\text{Left side of \eqref{thm21}}=\frac{P}{J_1}\tT(q),\label{3_3}
\end{align}
where
\begin{align}
\tT(q):=\sum_{c\in \bZ}\sum_{a, b\geq 0}\frac{q^{a^2+ab+b^2+a+b+bc+c^2+c}}{(q; q)_a(q; q)_b(-q; q)_c}.
\end{align}
Then by \eqref{II2} we have
\begin{align*}
P\cdot \tT(q)=\sum_{c\in \bZ}\sum_{a, b, d\geq 0}\frac{q^{a^2+ab+b^2+a+b+d^2/2+d/2+c^2+bc+dc+c}}{(q; q)_a(q; q)_b(q; q)_d}.
\end{align*}
Moreover, by setting $q=t^4$ we have
\begin{align}
P\cdot \tT(t^4)=\sum_{a, b, d\geq 0}\frac{t^{4a^2+4ab+3b^2-2bd+d^2+4a+2b}}{(t^4; t^4)_a(t^4; t^4)_b(t^4; t^4)_d}\sum_{c\in \bZ}t^{(2c+b+d+1)^2-1},
\end{align}
since the exponent identity
\begin{align*}
4a^2&+4ab+4b^2+4a+4b+2d^2+2d+4c^2+4c(b+d+1)\\
&=4a^2+4ab+3b^2-2bd+d^2+4a+2b+(2c+b+d+1)^2-1.
\end{align*}
This implies that
\begin{align}
P\cdot \tT(t^4)=t^{-1}(f(t^4, t^4)\tZ_o(t)+tf(1, t^8)\tZ_e(t)),\label{3_1}
\end{align}
where $\tZ_e(t)$ and $\tZ_o(t)$ are the even and odd $b+d$ parts of 
\begin{align}
\tZ(t):=\sum_{a, b, d\geq 0}\frac{t^{4a^2+4ab+3b^2-2bd+d^2+4a+2b}}{(t^4; t^4)_a(t^4; t^4)_b(t^4; t^4)_d}.
\end{align}
At this point, if we assume that Conjecture \ref{conj:SW14} holds, then we have
\begin{align}
\tZ(t)=\frac{
  (t^6; t^6)_\infty (t^8; t^8)_\infty (t^2, t^{10}; t^{12})_\infty
}{
  (t^4; t^4)_\infty^2 (t, t^{11}; t^{12})_\infty (t^5, t^7; t^{12})_\infty
}=\frac{J_2}{J_1^2}(f(t^8, t^{16})+tf(t^4, t^{20})).\label{3_2}
\end{align}
Substituting \eqref{3_2} into \eqref{3_1} gives
\begin{align}
P\cdot \tT(t^4)=\frac{J_2}{J_1^2}(f(t^4, t^4)f(t^4, t^{20})+f(1, t^8)f(t^8, t^{16}))=3\frac{J_2J_3^3}{J_1^3}.
\end{align}
Hence, \eqref{3_3} implies that
\begin{align}
\text{Left side of \eqref{thm21}}=3\frac{J_2J_3^3}{J_1^4}.
\end{align}

The above derivation shows that \eqref{thm21} would follow immediately if we could prove Conjecture \ref{conj:SW14}. This raises a very interesting question, namely, to investigate the relationship between these two types of identities, especially since we already have fully proved examples.

On the other hand, it is well known that the Rogers-Ramanujan identities \eqref{id:RR} are closely related to integer partitions. For the definition and details of integer partitions, we refer the reader to \cite{andtp}. Therefore, another natural and interesting question is how these identities can be interpreted from the perspective of integer partitions, whether they admit combinatorial proofs, or whether they point to some underlying connections.




\begin{thebibliography}{99}

\bibitem{andtp} G.E.~Andrews, {\it The theory of partitions}, Addison-Wesley, Reading, MA, 1976; reissued: Cambridge University Press, Cambridge, 1998.

\bibitem{ber06} B.C.~Berndt, {\it Number theory in the spirit of Ramanujan}, American Mathematical Society, Providence, RI, 2006.

\bibitem{BBG94} J.M.~Borwein, P.B.~Borwein and F.G.~Garvan, Some cubic modular identities of Ramanujan, Trans. Amer. Math. Soc. {\bf 343} (1994), 35--47.

\bibitem{CMP16} C.~Calinescu, A.~Milas and M.~Penn, Vertex algebraic structure of principal subspaces of basic $A_{2n}^{(2)}$-modules, J. Pure Appl. Algebra {\bf 220} (2016), 1752--1784.

\bibitem{CRW24} Z.~Cao, H.~Rosengren and L.~Wang, On some double Nahm sums of Zagier, J. Combin. Theory Ser. A {\bf 202} (2024), Paper No. 105819, 13 pp.

\bibitem{CW24} Z.~Cao and L.~Wang, {\it Some new modular rank three Nahm sums from a lift-dual operation}, \href{https://arxiv.org/abs/2412.15767v1}{arXiv:2412.15767v1}.

\bibitem{CW25} Z.~Cao and L.~Wang, {\it Some new modular rank four Nahm sums as lift-dual of rank three examples}, \href{https://arxiv.org/abs/2508.12468v1}{arXiv:2508.12468v1}.

\bibitem{CF13} I.~Cherednik and B.~Feigin, Rogers-Ramanujan type identities and Nil-DAHA, Adv. Math. {\bf 248} (2013), 1050--1088.

\bibitem{GR90} G.~Gasper and M.~Rahman, {\it Basic hypergeometric series}, Encyclopedia of Mathematics and its applications 35, Cambridge University Press, Cambridge, 1990.

\bibitem{LSZ08} J.~G. Mc~Laughlin, A.~V. Sills and P.~Zimmer, Rogers-Ramanujan-Slater type identities, Electron. J. Combin. {\bf DS15} (2008), Dynamic Surveys, 59 pp.

\bibitem{nah941} W.~Nahm, Conformal field theory and the dilogarithm, in: 11th International Conference on Mathematical Physics (ICMP-11) (Satelite Colloquia: New Problems in General Theory of Fields and Particles), Paris, 1994, pp. 662--667.

\bibitem{nah942} W.~Nahm, Conformal field theory, dilogarithm and three dimensional manifold, in: W. Nahm, J.-M. Shen (Eds.), Interface Between Physics and Mathematics, Proceedings, Conference in Hangzhou, People's Republic of China, September 1993, World Scientific, Singapore, 1994, pp. 154--165.

\bibitem{nah07} W. Nahm, Conformal field theory and torsion elements of the Bloch group, in: Frontiers in Number Theory, Physics and Geometry, II, Springer, 2007, pp. 67--132.

\bibitem{ram14} S.~Ramanujan, Problem 584, J. Indian Math. Soc. {\bf 6} (1914), 199--200.

\bibitem{ram19}S.~Ramanujan, Proof of certain identities in combinatory analysis, Proc. Camb, Philos. Soc. {\bf 19} (1919), 214--216.

\bibitem{rog94} L.J.~Rogers, Third memoir on the expansion of certain infinite products, Proc. Lond. Math. Soc. {\bf 26} (1894), 15--32.

\bibitem{sch17} I.~Schur, Ein Beitrag zur additiven Zahlentheorie und zur Theorie der Kettenbr\"{u}che, Sitzungsber. Preuss. Akad. Wiss. Phys.-Math. Klasse (1917), 302--321.

\bibitem{SW25} C.~Shi and L.~Wang, Modularity of tadpole Nahm sums in ranks 4 and 5, J. Number Theory {\bf 284} (2026), 214--245.

\bibitem{sil18} A.V.~Sills, {\it An invitation to the Rogers-Ramanujan identities}, CRC Press, Boca Raton (2018).

\bibitem{VZ11} M.~Vlasenko and S.~Zwegers, Nahm's conjecture: asymptotic computations and counterexamples, Commun. Number Theory Phys. {\bf 5}(3) (2011), 542--617.

\bibitem{wan241} L.~Wang, Identities on Zagier's rank two examples for Nahm's conjecture, Res. Math. Sci. {\bf 11}(3) (2024), Paper No. 49, 34 pp.

\bibitem{wan242} L.~Wang, Explicit forms and proofs of Zagier's rank three examples for Nahm's problem, Adv. Math. {\bf 450} (2024), Paper No. 109743, 46 pp.

\bibitem{zag07} D.~Zagier, The dilogarithm function, in: Frontiers in Number Theory, Physics and Geometry, II, Springer-Verlag, 2007, pp. 3--65.

\end{thebibliography}
\end{document}